\documentclass[12pt]{amsart}
\usepackage{scrextend}
\usepackage{amsmath,amsfonts,amssymb,amsthm,url,verbatim}
\usepackage[dvips]{graphicx}
\usepackage[croatian,english]{babel}
\usepackage[utf8]{inputenc}
\usepackage[T1]{fontenc}
\usepackage{latexsym}
\usepackage[usenames,dvipsnames]{xcolor}
\usepackage{soul}
\usepackage{booktabs}
\usepackage{longtable}
\usepackage{multirow}
\usepackage{ifthen}
\usepackage{graphpap}
\usepackage{xparse}
\definecolor{dark-blue}{rgb}{0.15,0.15,0.4}
\definecolor{medium-blue}{rgb}{0,0,0.5}
\definecolor{dark-red}{rgb}{0.4,0.15,0.15}
\usepackage[
	pagebackref=true,
	unicode=true,
	breaklinks=true,
	pdfsubject={doktorska disertacija},
	pdfstartpage={1},
	pdfstartview=FitH,
	linktocpage=false,
	bookmarksopen=true,
	bookmarksopenlevel=1,
	bookmarksnumbered=false,
	colorlinks=true,
	linkcolor={dark-blue},
	citecolor={medium-blue},
	urlcolor={dark-red},
	filecolor=yellow,
	]{hyperref}
\apptocmd{\sloppy}{\hbadness 10000\relax}{}{}
\renewcommand*{\backrefalt}[4]{
	\ifcase #1 %
	\else $\uparrow$ {\footnotesize #2.}%
	\fi
	}

\makeatletter
\newcommand\footnoteref[1]{\protected@xdef\@thefnmark{\ref{#1}}\@footnotemark}
\makeatother

\newtheorem{tm}{Theorem}[section]
\newtheorem{proposition}[tm]{Proposition}
\newtheorem{lemma}[tm]{Lemma}
\newtheorem{corollary}[tm]{Corollary}

\newtheorem*{acknowledgments}{Acknowledgments}
\theoremstyle{definition}

\theoremstyle{remark}

\newtheorem*{remark}{Remark}

\DeclareMathOperator{\@Gal}{Gal}
\DeclareMathOperator{\Img}{Img}

\newcommand\legendre[2]{\left(\frac{#1}{#2}\right)}

\newcommand{\Z}{\mathbb{Z}}

\newcommand{\N}{\mathbb{N}}
\newcommand{\F}{\mathbb{F}}
\newcommand{\Qbar}{{\overline{\mathbb Q}}}
\newcommand*{\tors}{\text{tors}}
\DeclareMathOperator{\GL}{GL}

\newcommand{\K}{\mathbb{K}}
\newcommand{\KK}{\mathcal{K}}
\newcommand{\Kpet}{\mathcal{K}_{\geq 5}}

\newcommand{\Aut}{\operatorname{Aut}}

\def\diam#1{\langle#1\rangle}
\NewDocumentCommand{\Q}{oo}{%
	\IfValueTF{#2}%
		{\mathbb{Q}_{#1,#2}}%
		{\IfValueTF{#1}%
			{\mathbb{Q}_{\infty,#1}}%
			{\mathbb{Q}}%
		}%
}
\NewDocumentCommand{\Gal}{oo}{%
	\IfValueTF{#2}%
		{\@Gal\ensuremath{\left(#1/#2\right)}}%
		{\IfValueTF{#1}%
			{}%
			{\@Gal}%
		}%
}
\NewDocumentCommand{\Qz}{oo}{%
	\IfValueTF{#2}%
		{\mathbb{Q}(\ensuremath{\mu}_{{#1}^{#2}})}%
		{\IfValueTF{#1}%
			{\mathbb{Q}(\ensuremath{\mu}_{#1})}%
			{\mathbb{Q}(\ensuremath{\mu}_{p^\infty})}%
		}%
}
\title{Torsion groups of elliptic curves over some infinite abelian extensions of $\mathbb{Q}$}
\author{Tomislav Gu\v{z}vi\'{c}}
\address{Department od Mathematics, University of Zagreb, Bijeni\v{c}ka cesta 30, 10000 Zagreb, Croatia}
\email{tguzvic@math.hr}
\urladdr{https://web.math.pmf.unizg.hr/~tguzvic/}
\author{Ivan Krijan}
\address{Department of Mathematics, University of Zagreb, Bijeni\v{c}ka cesta 30, 10000 Zagreb, Croatia}
\email{ikrijan@math.hr}
\urladdr{http://web.math.pmf.unizg.hr/~ikrijan/}
\thanks{Authors were supported by the QuantiXLie Center of Excellence, a project co-financed by the Croatian Government and European Union through the European Regional Development Fund - the Competitiveness and Cohesion Operational Programme (Grant KK.01.1.1.01.0004).}
\date{\today}

\begin{document}
\begin{abstract}
We determine, for an elliptic curve $E/\Q$, all the possible torsion groups \(E(\KK)_\tors\), where \(\KK\) is the compositum of all $\Z_p$-extensions of $\Q$. Furthermore, we prove that for an elliptic curve $E/\Q$ it holds that $E(\Qz)_\tors = E(\Qz[p])_\tors$, for all primes $p \geq 5$ and $E(\Qz[3][\infty])_\tors = E(\Qz[3][3])_\tors$, $E(\Qz[2][\infty])_\tors = E(\Qz[2][4])_\tors$.
\end{abstract}
\maketitle
\section{Introduction}
For a prime number \(p\), denote by \(\Q[p]\) the unique \(\Z_p\)-extension of \(\Q\), and for a positive integer \(n\), denote by \(\Q[n][p]\) the \(n\)\textsuperscript{th} layer of \(\Q[p]\), i.e.\ the unique subfield of \(\Q[p]\) such that $\Gal[\Q[n][p]][\Q]\simeq \Z/p^n\Z$.
	Recall that the \(\Z_p\)-extension of \(\Q\) is the unique Galois extension \(\Q[p]\) of $\Q$ such that
	\[\Gal[\Q[p]][\Q] \simeq \Z_p,\]
	where \(\Z_p\) is the additive group of the \(p\)-adic integers and is constructed as follows. Let
	\[G = \Gal[\Q(\zeta_{p^\infty})][\Q] = \varprojlim\limits_{n}\Gal[\Q(\zeta_{p^{n+1}})][\Q] \stackrel{\sim}{\rightarrow} \varprojlim\limits_{n}(\Z/p^{n+1}\Z)^\times = \Z_p^\times.\]
	Here we know that \(G = \Delta \times \Gamma\), where \(\Gamma \simeq \Z_p\) and \(\Delta \simeq \Z/(p-1)\Z\) for $p\geq 3$ and  \(\Delta \simeq \Z/2\Z\) (generated by complex conjugation) for $p=2$, so we define
	\[\Q[p] := \Q(\zeta_{p^\infty})^\Delta.\]
	We also see that every layer is uniquely determined by
	\[\Q[n][p] = \Q(\zeta_{p^{n+1}})^\Delta,\]
	so for $p\geq 3$ it is the unique subfield of \(\Q(\zeta_{p^{n+1}})\) of degree \(p^n\) over \(\Q\). More details and proofs of these facts about \(\Z_p\)-extensions and Iwasawa theory can be found in \cite[Chapter 13]{washington}.

Iwasawa theory for elliptic curves (see \cite{greenberg2}) studies elliptic curves in $\Z_p$-extensions, in particular the growth of the rank and $n$-Selmer groups in the layers of the $\Z_p$-extensions.

In the paper \cite{CDKN} we can find completely solved problem of determining how the torsion of an elliptic curve defined over $\Q$ grows in the $\Z_p$-extensions of $\Q$. In this paper we go a step further.

Let $\Kpet$ be the compositum of all $\Z_p$-extensions of $\Q$, for $p \geq 5$, i.e.
\[\Kpet = \prod\limits_{p \geq 5 \text{ prime}}\Q[p]\]
and let $\KK$ be the compositum of all $\Z_p$-extensions of $\Q$, i.e.
\[\KK = \prod\limits_{p \text{ prime}}\Q[p].\]

Our results are the following.

\begin{tm}\label{teo:nad_5_nista}
Let $E/\Q$ be an elliptic curve, then
\[E(\Kpet)_\tors = E(\Q)_\tors.\]
\end{tm}

\begin{tm}\label{teo:na_2_i_3_isto_kao_i_ranije}
Let $E/\Q$ be an elliptic curve, then $E(\KK)_\tors$ is exactly one of the following groups:
\begin{align*}
\Z / n\Z, &\quad 1 \leq n \leq 10 \text{ or } n \in \{12, 13, 21, 27\}, \\
\Z / 2\Z \oplus \Z/2n\Z, &\quad 1 \leq n \leq 4.
\end{align*}
For each group $G$ from the list above there exists an $E/\Q$ such that \(E(\KK)_\tors \simeq G\).
\end{tm}

\begin{remark}
By Mazur's theorem \cite{mazur2} we see that
\[\{E(\KK)_\tors : E/\Q \text{ elliptic curve}\} = \{E(\Q)_\tors : E/\Q \text{ elliptic curve}\} \cup \{\Z/13\Z, \Z/21\Z, \Z/27\Z\}.\]
	However, there are many elliptic curves $E/\Q$ for which torsion grows from $\Q$ to $\KK$. In \cite{CDKN} one can find many examples of torsion growth $\Q \to \Q[2]$ and $\Q \to \Q[3]$. Same examples apply here. At the and of the proof of the Theorem \ref{teo:na_2_i_3_isto_kao_i_ranije} (more precisely, in the Lemma \ref{lem:kompozitum_13}) we show that there are some elliptic curves $E/\Q$ such that $E(\KK)_\tors \simeq \Z/13\Z$.
\end{remark}

Furthermore, let $\mu_n$ be the set of all complex numbers $\omega$ such that $\omega^n = 1$. Note that for a prime number $p$ we have that $\Q(\mu_p) = \Q(\zeta_p)$, where $\zeta_p$ is, as usual, $p$\textsuperscript{th} primitive root of unity.

For a prime number $p$, we define a set $\mu_{p^\infty}$ as the set of all complex numbers $\omega$ for which there exists non-negative integer $k$ such that $\omega^{p^k} = 1$. Note that $\Qz$ is the set $\Q$ extended with all $p^{n^{\text{th}}}$ primitive roots of unity.

\begin{tm} \label{teo:rast_qzetap}
Let $E/\Q$ be an elliptic curve, then for a prime number $p \geq 5$ it holds that
\[E(\Qz)_\tors = E(\Qz[p])_\tors.\]
Furthermore,
\[E(\Qz[3][\infty])_\tors = E(\Qz[3^3])_\tors \qquad \text{and} \qquad E(\Qz[2][\infty])_\tors = E(\Qz[2^4])_\tors.\]
\end{tm}
\begin{remark}
This situation is ``the best possible''. For $E = \href{http://www.lmfdb.org/EllipticCurve/Q/27a4}{27a4}$ (Cremona reference from \cite{lmfdb}) we have that
\[E(\Qz[3^2])_\tors = \Z/9\Z \subsetneq \Z/27\Z = E(\Qz[3^3])_\tors\]
and for $E = \href{http://www.lmfdb.org/EllipticCurve/Q/32a4}{32a4}$ it holds that
\[E(\Qz[2^3])_\tors = \Z/2\Z \oplus \Z/4\Z \subsetneq \Z/2\Z \oplus \Z/8\Z = E(\Qz[2^4])_\tors.\]
\end{remark}

\section{Notation and auxiliary results}
In this paper we deal with elliptic curves defined over $\Q$, so unless noted otherwise, all elliptic curves will be assumed to be defined over $\Q$.

We will use standard notation regarding Galois representations attached to elliptic curves throughout the paper following \cite{silverman}. More details can be found in \cite[\S 2]{CDKN}.

To make this paper as self-contained as reasonably possible, we now list the most important known results that we will use.

Firstly, we state main results from \cite{CDKN}:
	\begin{tm}\label{bigpresult}
	Let \(p \geq 5\) be a prime number, and $E/\Q$ an elliptic curve. Then
	\[E(\Q[p])_\tors = E(\Q)_\tors.\]
	\end{tm}
	\begin{tm}\label{p2result}
	Let $E/\Q$ be an elliptic curve. \(E(\Q[2])_\tors\) is exactly one of the following groups:
	\begin{align*}
	\Z/N\Z, &\qquad 1 \leq N \leq 10 \text{, or } N = 12,\\
	\Z/2\Z \oplus \Z/2N\Z, &\qquad 1 \leq N \leq 4,
	\end{align*}
	\end{tm}
	\begin{tm}\label{teo:p_jednak_3}
	Let $E/\Q$ be an elliptic curve. \(E(\Q[3])_\tors\) is exactly one of the following groups:
	\begin{align*}
	\Z/N\Z, &\qquad 1 \leq N \leq 10 \text{, or } N = 12, 21 \text{ or } 27,\\
	\Z/2\Z \oplus \Z/2N\Z, &\qquad 1 \leq N \leq 4.
	\end{align*}
		and for each group $G$ from the list above there exists an $E/\Q$ such that \(E(\Q[3])_\tors  \simeq G.\)
	\end{tm}

Furthermore, we state known results that we will be using through this paper. Most of these results are stated in \cite{CDKN}, but we will state them here also (although without the proofs), for the sake of completeness.

\begin{proposition}\cite[Ch.\ III, Cor.\ 8.1.1]{silverman} \label{pro:weil} Let $E/L$ be an elliptic curve with $L\subseteq\Qbar$. For each integer $n\geq 1$, if $E[n]\subseteq E(L)$ then the $n$\textsuperscript{th} cyclotomic field $\Q(\zeta_n)$ is a subfield of $L$.
\end{proposition}

An immediate consequence of this proposition is

\begin{corollary} \label{cor-weil}
Let \(p\) and \(q\) be odd prime numbers. Then
\[E(\Q[p])[q] \simeq \{O\} \quad \text{or} \quad \Z/q\Z.\]
\end{corollary}

\begin{remark}
We have that \(E[q^n] \nsubseteq E(\Q[p])\), for each positive integer \(n\).
\end{remark}

\begin{lemma}\cite[Lemma 4.6]{dlns} \label{lem-j-k_isog}
Let $E$ be an elliptic curve over a number field $K$, let $F$ be a Galois extension of $\Q$, let $p$ be a prime, and let $k$ be the largest integer for which $E[p^k]\subseteq E(F)$.
If $E(F)_\tors$ contains a subgroup isomorphic to $\Z/p^k\Z\oplus\Z/p^j\Z$ with $j\ge k$, then~$E$ admits a $K$-rational $p^{j-k}$-isogeny.
\end{lemma}

\begin{tm}\label{teo:mazur_izogenije}\cite{mazur2,kenku2,kenku3,kenku4,kenku5}
Let $E/\Q$ be an elliptic curve with a rational $n$-isogeny. Then
\[
n\leq 19 \text{ or } n \in\{21,25,27,37,43,67,163\}.
\]
\end{tm}

\begin{corollary}\label{kor:ubij_torziju}
Let \(p\) be an odd prime number, \(E/\Q\) elliptic curve and \(P \in E(\Q[p])_\tors\) a point of order \(q^n\) for some prime \(q\) and positive integer \(n\), then
\[q^n \in \{2, 3, 4, 5, 7, 8, 9, 11, 13, 16, 17, 19, 25, 27,32, 37, 43, 67, 163\}.\]
\end{corollary}

\begin{tm}\label{teo:rast_torzije}\cite[Theorem 5.7]{gn}
Let $E/\Q$ be an elliptic curve, p a prime and $P$ a point of order $p$ on $E$. Then all of the cases in the table below occur for $p\leq 13$ or $p=37$, and they are the only ones possible.
\[\begin{array}{|c|c|}
\hline
p & [\Q(P):\Q]\\
\hline
2 & 1,2,3\\
\hline
3 & 1,2,3,4,6,8\\
\hline
5 & 1,2,4,5,8,10,16,20,24\\
\hline
7 & 1,2,3,6,7,9,12,14,18,21,24,36,42,48\\
\hline
11 & 5,10,20,{ 40},55,{ 80},100,110,120\\
\hline
13 & 3,4,6,12,{ 24},39,{ 48},52,72,78,96,{ 144},156,168\\

\hline
37 & 12,36,{ 72},444,{ 1296},1332, 1368\\
\hline
\end{array}\]
For all other $p$, for $[\Q(P):\Q]$ the following cases do occur:
\begin{enumerate}
\item \(p^2 - 1\), \hfill for all \(p\),
\item \(8,\ 16,\ 32,\ 136,\ 256,\ 272,\ 288\), \hfill for \(p = 17\),
\item \(\displaystyle \frac{p - 1}{2},\ p-1,\ \frac{p(p-1)}{2},\ p(p-1)\), \hfill if \(p \in \{19,43,67,163\}\),

\item \(2(p-1),\ (p-1)^2\), \hfill if \(p \equiv 1 \pmod{3}\) or \(\legendre{-D}{p} = 1\),\\
							\mbox{} \hfill for some \(D \in \{1,2,7,11,19,43,67,163\}\),
\item \(\displaystyle \frac{(p-1)^2}{3},\ \frac{2(p-1)^2}{3}\), \hfill if \(p \equiv 4, 7 \pmod{9}\),
\item \(\displaystyle \frac{p^2-1}{3},\ \frac{2(p^2 - 1)}{3}\), \hfill if \(p \equiv 2, 5 \pmod{9}\),
\end{enumerate}

Apart from the cases above that have been proven to appear, the only other options that might be possible are:
\[\frac{p^2 - 1}{3},\ \frac{2(p^2 - 1)}{3}, \quad \text{ for } p \equiv 8 \pmod{9}.\]
\end{tm}

\begin{tm}	\cite[Theorem 7.2.]{gn}\label{l-tors}
Let \(p\) be the smallest prime divisor of a positive integer \(d\) and let \(K/\Q\) be a number field of degree \(d\).
\begin{itemize}
\item If \(p \geq 11\), then \(E(K)_{\tors} = E(\Q)_{\tors}\).
\item If \(p = 7\), then \(E(K)[q^\infty] = E(\Q)[q^\infty]\), for all primes \(q \neq 7\).
\item If \(p = 5\), then \(E(K)[q^\infty] = E(\Q)[q^\infty]\), for all primes \(q \neq 5, 7, 11\).
\item If \(p = 3\), then \(E(K)[q^\infty] = E(\Q)[q^\infty]\), for all primes \(q \neq 2,3,5,7,11,13,19,43,67,163\).
\end{itemize}
\end{tm}

We now prove a lemma that we will find useful.

\begin{lemma} \label{ciklicko_prosirenje}
Let $p$ and $q$ be prime numbers such that $q - 1 \nmid p$ and $p \nmid q - 1$. Let $K/\Q$ be a cyclic extension of degree $p$, and $P \in E$ a point of degree $q$. If $P\in E(K)$, then $P \in E(\Q)$.
\end{lemma}

The following lemma will tell us how far up the tower we have to go to find a point of order $n$, if such a point exists.

\begin{lemma} \label{lem:sloj_tocke}

Let $E/\Q$ be an elliptic curve and $P\in  E$ a point of order $n$ such that $\Q(P)/\Q$ is Galois and let $E(\Q(P))[n]\simeq \Z/n\Z$. Then $\Gal(\Q(P)/\Q)$ is isomorphic to a subgroup of $(\Z/n\Z)^\times$.
\end{lemma}

We immediately obtain the following corollary.

\begin{corollary}
\label{kojepolje}
Let $P\in  E$ be a point of order $n$ such that $\Q(P)\subseteq \Q_{\infty,p}$. Then $\Q(P)\subseteq \Q_{m,p}$, where $m=v_p(\phi(n))$.
\end{corollary}

\begin{proposition}
\label{prop_pdiv}
Let $E/F$ be an elliptic curve  over a number field $F$, $n$ a positive integer, $P \in E$ be a point of order $p^{n+1}$ such that $E(F(pP))$ has no points of order $p^{n+1}$ and such that $F(P)/F(pP)$ is Galois. Then $[F(P):F(pP)]$ divides $p^2$.
\end{proposition}

\begin{remark}
Proposition \ref{prop_pdiv} is a version of \cite[Proposition 4.6.]{gn} with stronger assumptions.
\end{remark}

\section{Proof of Theorem \ref{teo:nad_5_nista}}
We know that $\Gal[\Kpet][\Q] \simeq \prod\limits_{p \geq 5 \text{ prime}}\Z_p$. Therefore, we see that the following holds:

Let $\F$ be a number field contained in $\Kpet$ and let $[\F : \Q] = d$. If $p$ is the smallest prime divisor of $d$, then $p \geq 5$. Any point in $E(\Kpet)_\tors$ is defined over some finite extension $\F$ of $\Q$, where $\Q \subseteq \F \subseteq \Kpet$. Using the Theorem \ref{teo:rast_torzije} we can see that \[E(\Kpet)[q^\infty] = E(\Q)[q^\infty],\] for all prime numbers $q$ different from $5$, $7$ and $11$. Therefore, it remains to show that
\[E(\Kpet)[5^\infty] = E(\Q)[5^\infty], \quad E(\Kpet)[7^\infty] = E(\Q)[7^\infty] \quad \text{and} \quad E(\Kpet)[11^\infty] = E(\Q)[11^\infty],\]
which is proven by the following three Lemmas.

Before stating and proving these Lemmas, let us mention that the Corollary \ref{kor:ubij_torziju} holds for the field $\Kpet$. It is enough to notice that $\Kpet$ is an cyclic extension of $\Q$ and therefore so are all subfields of $\Kpet$. Furthermore, the only roots of unity contained in $\Kpet$ are $\pm 1$, so we have all the necessary ingredients for proving Corollary \ref{kor:ubij_torziju} in the case when the field is $\Kpet$.

\begin{lemma}\label{lem:nad_Kpet_nema_5}
Let $E /\Q$ be an elliptic curve. Then
\[E(\Kpet)[5^\infty] = E(\Q)[5^\infty].\]
\end{lemma}
\begin{proof}
Keeping in mind the above observation, we can see that the proof is analogous to the proof of \cite[Lemma 3.6.]{CDKN}.
\end{proof}
\begin{lemma}\label{lem:nad_Kpet_nema_7}
Let $E /\Q$ be an elliptic curve. Then
\[E(\Kpet)[7^\infty] = E(\Q)[7^\infty].\]
\end{lemma}
\begin{proof}
By the Corollary \ref{kor:ubij_torziju} we have that there are no points of order $49$ in $E(\Kpet)$. It remains to show that $E(\Kpet)[7] = E(\Q)[7]$. Theorem \ref{teo:rast_torzije} implies that the only possibility is $P \in E(\Q[1][7])$, where $P \in E(\Kpet)$ is a point of order $7$. As we did in the proof of \cite[Theorem 3.1]{CDKN}, we conclude that $P \in E(\Q)$.
\end{proof}

Before proving that $E(\Kpet)[11^\infty] = E(\Q)[11^\infty]$, we shall state and prove one technical consequence of Lemma \ref{lem:sloj_tocke} that we shall find useful also in the proof of Theorem \ref{teo:na_2_i_3_isto_kao_i_ranije}..
\begin{corollary}\label{kor:izogenija_stupanj}
Let $E /\Q$ be an elliptic curve and let $P \in E$ be a point of order $n$. If $E$ has a rational $n$-isogeny with kernel $\langle P \rangle$, then $[\Q(P) : \Q] \mid \phi(n)$.
\end{corollary}

\begin{proof}
This follows directly from Lemma \ref{lem:sloj_tocke}, since the group $(\Z/n\Z)^\times$ has exactly $\phi(n)$ elements.
\end{proof}

\begin{lemma}\label{lem:nad_Kpet_nema_11}
Let $E /\Q$ be an elliptic curve. Then
\[E(\Kpet)[11^\infty] = \{0\}.\]
\end{lemma}
\begin{proof}
By the Corollary \ref{kor:ubij_torziju} we see that $E(\Kpet)$ does not contain a point of order $121$. Therefore, it remains to show that $E(\Kpet)[11] = \{0\}$. Assume that $P \in E(\Kpet)$, where $P$ is a point of order $11$. Combining the Theorem \ref{teo:rast_torzije} with the fact that $\Q \subseteq \Q(P) \subseteq \Kpet$, we see that $[\Q(P) : \Q] \in \{5, 55\}$. Using Corollary \ref{kor:izogenija_stupanj} we see that $[\Q(P) : \Q] = 5$. But this implies that $\Q(P) = \Q[1][5]$. The rest of the proof is the same as the proof of \cite[Lemma 3.2.]{CDKN}.
\end{proof}
\section{Proof of Theorem \ref{teo:na_2_i_3_isto_kao_i_ranije}}
Proving this theorem will require a little more work. Namely, the approach used in the Theorem \ref{teo:nad_5_nista} will not help us in this case, since finite subextension $\Q \subseteq \F \subseteq \KK$ can be of arbitrary degree over $\Q$.
First we shall mention the main result from the paper \cite[Theorem 1.2.]{chou}, whose author is Michael Chou.
\begin{tm}\label{teo:chou_max_abel}
Let $E/\Q$ be an elliptic curve and let $\Q^\text{ab}$ be a maximal Abelian extension of $\Q$. The group $E(\Q^\text{ab})_\tors$ is isomorphic to one of the following groups:
\begin{align*}
\Z/N_1\Z, &\qquad N_1 = 1, 3, 5, 7, 9, 11, 13, 15, 17, 19, 21, 25, 27, 37, 43, 67, 163, \\
\Z/2\Z \oplus \Z/2N_2\Z, &\qquad N_2 = 1, 2, 3, 4, 5, 6, 7, 8, 9, \\
\Z/3\Z \oplus \Z/3N_3\Z, &\qquad N_3 = 1, 3, \\
\Z/4\Z \oplus \Z/4N_4\Z, &\qquad N_4 = 1, 2, 3, 4, \\
\Z/5\Z \oplus \Z/5\Z, \\
\Z/6\Z \oplus \Z/6\Z, \\
\Z/8\Z \oplus \Z/8\Z.
\end{align*}
For every group $G$ in the list above, there exists an elliptic curve $E/\Q$ such that $E(\Q^\text{ab})_\tors \simeq G$.
\end{tm}
Notice that $E(\KK)$ does not contain a full $\Z/n\Z \oplus \Z/n\Z$ torsion for any $n > 2$. We also know that $\KK \subseteq \Q^\text{ab}$. Combining these two facts and the Theorem \ref{teo:chou_max_abel}, we conclude that the group $E(\KK)_\tors$ is isomorphic to one of the following groups:
\begin{align*}
\Z/n\Z, &\quad 1 \leq n \leq 19 \text{ or } n \in \{21, 25, 27, 37, 43, 67, 163\}, \\
\Z/2\Z \oplus \Z/2n\Z, &\quad 1 \leq n \leq 9.
\end{align*}
Now we see that in order to prove the Theorem \ref{teo:na_2_i_3_isto_kao_i_ranije}, we have to show that $E(\KK)_\tors$ does not contain points of order 
\[11,\ 14,\ 15,\ 16,\ 17,\ 18,\ 19,\ 25,\ 37,\ 43,\ 67,\ 163,\]
and that if $E(\KK)_\tors$ contains a point of order $10$ or $12$, then $E(\KK)_\tors \simeq \Z/10\Z$ or $E(\KK)_\tors \simeq \Z/12\Z$, respectively. We shall prove this with a series of Lemmas. Finally, using the Lemma \ref{lem:kompozitum_13} we show that there exist elliptic curves $E/\Q$ that contain a point of order $13$ defned over $\KK$.
\begin{lemma}\label{lem:lagani_prosti}
Let $E /\Q$ be an elliptic curve and let $p \in \{11, 19, 37, 43, 67, 163\}$. Then
\[E(\KK)[p] = \{0\}.\]
\end{lemma}
\begin{proof}
Assume that $P \in E(\KK)$ is a point of order $p$. Using Lemma \ref{lem-j-k_isog} we know that $E$ has a rational $p$-isogeny with kernel $\diam P$. Corollary \ref{kor:izogenija_stupanj} implies that \[[\Q(P) : \Q] \mid \phi(p) = p - 1.\] It is known that such elliptic curve has CM. By the \cite[Table 3]{loz} we can see all the possibilities for $j$-invariants of $E$. For the convenience of the reader, we list them in the table below.
\begin{table}[h!]
\centering
\begin{tabular}{|c|c|c|} \hline
$p$		& $j$-invariant										& Cremona label											\\ \hline
$11$	& $-11 \cdot 131^3$										& \href{http://www.lmfdb.org/EllipticCurve/Q/121a1}{121a1}		\\
		& $-2^{15}$												& \href{http://www.lmfdb.org/EllipticCurve/Q/121b1}{121b1}		\\
		& $-11^2$												& \href{http://www.lmfdb.org/EllipticCurve/Q/121c1}{121c1}		\\ \hline
$19$	& $-2^{15} \cdot 3^3$									& \href{http://www.lmfdb.org/EllipticCurve/Q/361a1}{361a1}		\\ \hline
$37$	& $-7 \cdot 11^3$										& \href{http://www.lmfdb.org/EllipticCurve/Q/1225h1}{1225h1}	\\
		& $-7 \cdot 137^3 \cdot 2083^3$							& \href{http://www.lmfdb.org/EllipticCurve/Q/1225h2}{1225h2}	\\ \hline
$43$	& $-2^{18} \cdot 3^3 \cdot 5^3$ 						& \href{http://www.lmfdb.org/EllipticCurve/Q/1849a1}{1849a1}	\\ \hline
$67$	& $-2^{15} \cdot 3^3 \cdot 5^3 \cdot 11^3$				& \href{http://www.lmfdb.org/EllipticCurve/Q/4489a1}{4489a1}	\\ \hline
$163$	& $-2^{18} \cdot 3^3 \cdot 5^3 \cdot 23^3 \cdot 29^3$	& \href{http://www.lmfdb.org/EllipticCurve/Q/26569a1}{26569a1}	\\ \hline
\end{tabular}
\caption{Elliptic curves with $j$-invariant $j$ with minimal conductor \cite{lmfdb}}
\label{tablica:j_invarijante}
\end{table}
It is known that the division polynomials of elliptic curves with same $j$-invariant differ by a multiplicative constant. Therefore, the $p^{\text{th}}$ division polynomial of $E$ is the same (up to multiplication by a constant) as the $p^{\text{th}}$ division polynomial of some elliptic curve from the Table \ref{tablica:j_invarijante}.

Using \verb|magma| \cite{magma} we can easily calculate the $p^{\text{th}}$ division polynomial of elliptic curves given in the Table \ref{tablica:j_invarijante}. Since $[\Q(P) : \Q] \leq p - 1$, we are searching for the irreducible factors of these polynomials that are of degree $\leq p-1$ and we can easily see that they do not have zeroes over $\KK$. Namely, each zero of an irreducible (over $\Q$) polynomial of degree $d$ is defined over some degree $d$ number field. The degree $d$ number field contained in $\KK$ is unique and we can easily find it using \verb|magma| \cite{magma}. It remains to check that the polynomial does not have zeroes defined over that field which completes the proof in all cases except when $p=163$. Namely, in all other cases, \verb|magma| \cite{magma} computes what we need in a matter of seconds. The case when $p=163$ requires slightly different approach. As we can see in the Table \ref{tablica:polinomcici}, the only irreducible factor (call it $\varphi$) of the $163$\textsuperscript{rd} division polynomial whose degree is less then or equal to $163$ has a degree $81$. If we assume that $E(\KK)$ contains a point of order $163$, then we know that it's $x$-coordinate must be defined over a field of degree $81$, which is $\Q[4][3]$. But $\Q[4][3] / \Q$ is Galois extension and we know that over $\Q[1][3] = \Q(\zeta_9)^+$, $\varphi$ splits into a product of three irreducible factors, each having a degree $27$. Computation in \verb|magma| \cite{magma} shows that this is not the case. Namely, the polynomial $\varphi$ remains irreducible over $\Q(\zeta_9)^+$, which leads to a contradiction.
\end{proof}
\begin{table}[ht]
\centering
\begin{tabular}{|c|c|c|c|} \hline
$p$		& Elliptic curve										& $\deg \psi_p$	& irreducible factors of degree $\leq p - 1$	\\ \hline
$11$	& \href{http://www.lmfdb.org/EllipticCurve/Q/121a1}{121a1}		& $60$			& one of degree $5$ 						\\
		& \href{http://www.lmfdb.org/EllipticCurve/Q/121b1}{121b1}		& $60$			& one of degree $5$							\\
		& \href{http://www.lmfdb.org/EllipticCurve/Q/121c1}{121c1}		& $60$			& one of degree $5$							\\ \hline
$19$	& \href{http://www.lmfdb.org/EllipticCurve/Q/361a1}{361a1}		& $180$			& one of degree $9$ 						\\ \hline
$37$	& \href{http://www.lmfdb.org/EllipticCurve/Q/1225h1}{1225h1}	& $684$			& three of degree $6$							\\
		& \href{http://www.lmfdb.org/EllipticCurve/Q/1225h2}{1225h2}	& $684$			& one of degree $18$						\\ \hline
$43$	& \href{http://www.lmfdb.org/EllipticCurve/Q/1849a1}{1849a1}	& $924$			& one of degree $21$						\\ \hline
$67$	& \href{http://www.lmfdb.org/EllipticCurve/Q/4489a1}{4489a1}	& $2244$		& one of degree $33$						\\ \hline
$163$	& \href{http://www.lmfdb.org/EllipticCurve/Q/26569a1}{26569a1}	& $13284$		& one of degree $81$						\\ \hline
\end{tabular}
\caption{Possible degrees of irreducible factors of degree $\leq p - 1$}
\label{tablica:polinomcici}
\end{table}

\begin{proposition}{\cite[Proposition 4.6.]{gn}}\label{pro:gonzalo_najman}
Let $F$ be a number field and $E/F$ be an elliptic curve. Let $p$ be a prime number, $n \in \N$ and $P \in E(\overline{\F})$ a point of order $p^{n+1}$. Then $[\F(P) : \F(pP)]$ divides $p^2$ or $(p-1)p$.
\end{proposition}

\begin{lemma}\label{lem:kompozitum_16}
Let $E /\Q$ be an elliptic curve. Then $E(\KK)$ does not contain a point of order $16$.
\end{lemma}
\begin{proof}
Assume that $P \in E(\KK)$ is a point of order $16$. Lemma \ref{lem-j-k_isog} implies that $E$ has a rational $2$-isogeny. We also know that $E(\KK)_\tors$ is isomorphic to a subgroup of $\Z/2\Z \oplus \Z/16\Z$ by the Theorem \ref{teo:chou_max_abel}. By the \cite[Theorem 4]{siksek} we see that \[[\Q(E[2]) : \Q] = |G_{E,\Q}(2)| \leq |B(2)| = 2.\]
The point $8P$ is of order $2$ so we have $[\Q(8P) : \Q] \in \{1, 2\}$. By Proposition  \ref{pro:gonzalo_najman} we have \[[\Q(P) : \Q] = \underbrace{[\Q(P) : \Q(2P)]}_{\in \{1,2,4\}}\underbrace{[\Q(2P) : \Q(4P)]}_{\in \{1,2,4\}}\underbrace{[\Q(4P) : \Q(8P)]}_{\in \{1,2,4\}}\underbrace{[\Q(8P) : \Q]}_{\in \{1,2\}}.\] Therefore, $[\Q(P) : \Q]$ is a power of $2$ which means that $P$ is defined over $\Q[2]$, which contradicts the Theorem  \ref{p2result}.
\end{proof}
\begin{lemma}\label{lem:kompozitum_15_17}
Let $E /\Q$ be an elliptic curve and let $n \in \{15,17\}$. Then $E(\KK)$ does not contain a point of order $n$. 
\end{lemma}
\begin{proof}
Assume that $P \in E(\KK)$ is a point of order $n$. It follows that $E$ has a rational $n$-isogeny by the Lemma \ref{lem-j-k_isog}. By the Corollary \ref{kor:izogenija_stupanj} we see that 

\[[\Q(P) : \Q] \mid \phi(n).\]
On the other hand, we know that $\phi(15) = 8$ and $\phi(17) = 16$, which means that $P$ is defined over $\Q[2]$, which contradicts the Theorem \ref{p2result}.
\end{proof}
\begin{lemma}\label{lem:kompozitum_14_18}
Let $E /\Q$ be an elliptic curve and let $n \in \{14,18\}$. Then $E(\KK)$ does not contain a point of order $n$.
\end{lemma}
\begin{proof}
Notice that $n=2k$, where $k \in \{7,9\}$. Assume that $E(\KK)$ contains a point $P$ of order $n$. Then $E(\KK)$ has a point $P_k$ of order $k$. By the Lemma \ref{lem-j-k_isog} and the Corollary \ref{kor:izogenija_stupanj} we see that 
\[[\Q(P_k) : \Q] \mid \phi(k) = 6.\]
This means that $P_k$ is defined over a number field of degree at most $6$, so it's defined over the compositum $\Q[1][2]\Q[1][3]$. From the \cite[Corollary 4]{najman_twist} we get that 
\[E(\Q[1][2]\Q[1][3])[k] \simeq E(\Q[1][3])[k] \oplus E^{(2)}(\Q[1][3])[k],\] where $E^{(2)}$ is the quadratic twist of $E$ by $2$. We know that $E^{(2)}$ is defined over $\Q$ and that having a $2$-torsion is a twist invariant property. Since $kP$ is of order $2$, we know that $E$ and $E^{(2)}$ have at least one point of order $2$ defined over $\Q[1][3]$. This means that there exists an elliptic curve defined over $\Q$ that has a point of order $k$ and a point of order $2$ defined over $\Q[1][3]$, so it has a point of order $n=2k$ defined over $\Q[1][3]$, which contradicts the Theorem \ref{teo:p_jednak_3}.
\end{proof}

\begin{lemma}\label{lem:kompozitum_25}
Let $E /\Q$ be an elliptic curve. Then $E(\KK)$ does not contain a point of order $25$.
\end{lemma}
\begin{proof}
Assume that $P \in E(\KK)$ is a point of order $25$. Using the Lemma \ref{lem-j-k_isog} we see that $E$ has a rational $25$-isogeny, which means that $\Gal[\Q(E[25])][\Q]$ acts on $\diam{P}$, so for each $\sigma \in \Gal[\Q(E[25])][\Q]$ there exists some $a \in (\Z/25\Z)^\times$ such that $P^\sigma = aP$. By the Corollary \ref{kor:izogenija_stupanj} we see that $[\Q(P) : \Q] \mid \phi(25) = 20$. Furthermore, the point $5P$ is of order $5$ and by the same reasoning we conclude that $[\Q(5P) : \Q] \mid \phi(5) = 4$. Therefore, we have 
\[\Q(P) \subseteq \Q[2][2]\Q[1][5] \qquad \text{and} \qquad \Q(5P) \subseteq \Q[2][2].\]
Each $\sigma \in \Gal[\Q(E[25])][\Q[2][2]]$ fixes the point $5P$, so we conclude that $G_{\Q[2][2]}(25)$ is (up to conjugacy) of the form 
\[\left\lbrace\begin{pmatrix}a & \ast \\ 0 & \ast\end{pmatrix} : a \in 1 + 5\Z/25\Z\right\rbrace.\]
We also know that $[\Q(\zeta_{25}) : \Q[1][5]] = 4$, which means that $\left|\Gal[\Q(\zeta_{25})][\Q[1][5]]\right| = 4$. By the \cite[Lemma 3.4.]{CDKN} we have that $\det G_{\Q[2][2]\Q[1][5]}(25)$ is isomorphic to an unique order $4$ subgroup of $(\Z/25\Z)^\times$, which is $\diam{7} = \{7, -1, -7, 1\}$. Let us recall that $\Gal[\Qbar][\Q[2][2]\Q[1][5]]$ fixes the point $P$, so 
\[G_{\Q[2][2]\Q[1][5]}(25) \leq \left\lbrace\begin{pmatrix}1 & \ast \\ 0 & b\end{pmatrix} : b \in \{7, -1, -7, 1\} \right\rbrace.\]
Notice that $\Gal[\Q(E[25])][\Q[2][2]]$ does not fix the point $P$, by the Theorem \ref{p2result} and that 

\[[G_{\Q[2][2]}(25) : G_{\Q[2][2]\Q[1][5]}(25)] = [\Q[2][2]\Q[1][5] : \Q[2][2]] = 5.\]

Therefore we have that
\[G_{\Q[2][2]}(25) \leq \left\lbrace\begin{pmatrix}a & \ast \\ 0 & b\end{pmatrix} : a \in 1 + 5\Z/25\Z,\ b \in \{7, -1, -7, 1\}\right\rbrace.\]
We have that $[G_{\Q}(25) : G_{\Q[2][2]}(25)] = [\Q[2][2] : \Q] = 4$ and 
\[[(\Z/25\Z)^\times : 1 + 5\Z/25\Z] = 4 \qquad \text{i} \qquad [(\Z/25\Z)^\times : \diam{7}] = 5.\]
Finally we conclude that 
\[G_{\Q}(25) \leq \left\lbrace\begin{pmatrix}a & \ast \\ 0 & b\end{pmatrix} : a \in (\Z/25\Z)^\times,\ b \in \{7, -1, -7, 1\}\right\rbrace.\]
We compute

\[25 \mid 150 \mid [\GL_2(\Z/25\Z) : G_{\Q}(25)] \mid [\Aut_{\Z_5}(T_5(E)) : \Img(\overline{\rho}_{5, E})],\]
which contradicts the Theorem \cite[Theorem 2]{greenberg}.

\end{proof}
\begin{lemma}\label{lem:kompozitum_10}
Let $E/\Q$ be an elliptic curve such that $E(\KK)$ contains a point of order $10$. Then 
\[E(\KK)_\tors \simeq \Z/10\Z.\]

\end{lemma}
\begin{proof}
The only other possibility for $E(\KK)_\tors$ is the group $\Z/2\Z \oplus \Z/10\Z$, so we shall eliminate that possibility. For the sake of contradiction, assume that $E(\KK)_\tors \simeq \Z/2\Z \oplus \Z/10\Z$.

The entire $2$-torsion of $E$ is defined over $\KK$ which is Galois over $\Q$ so we conclude that $G_{\Q}(2)$ is conjugate to one of the following groups:

\[G_1 = \left\lbrace\begin{pmatrix}1 & 0 \\ 0 & 1\end{pmatrix}\right\rbrace, \quad G_2 = \left\lbrace\begin{pmatrix}1 & 0 \\ 0 & 1\end{pmatrix},\begin{pmatrix}1 & 1 \\ 0 & 1\end{pmatrix}\right\rbrace, \quad G_3 = \left\lbrace\begin{pmatrix}1 & 0 \\ 0 & 1\end{pmatrix},\begin{pmatrix}1 & 1 \\ 1 & 0\end{pmatrix},\begin{pmatrix}0 & 1 \\ 1 & 1\end{pmatrix}\right\rbrace.\]
In the case when $G_{\Q}(2) = G_1$ or $G_{\Q}(2) \cong G_2$ (up to conjugacy), we can immediately conclude that the entire $2$-torsion subgroup of $E$ is defined over at most quadratic extension of $\Q$, so it's defined over $\Q[1][2]$. Furthermore, using the Lemma \ref{lem-j-k_isog} and the Corollary \ref{kor:izogenija_stupanj} we see that the point of order $5$ is defined over $\Q$, $\Q[1][2]$ or and $\Q[2][2]$. We conclude that $E(\Q[2])_\tors$ contains $\Z/2\Z \oplus \Z/10\Z$, which contradicts the Theorem \ref{p2result}.
It remains to consider the case when $G_{\Q}(2) = G_3$. Using the \cite[Theorem 3.6.]{gn} we conclude that $E$ does not have CM. Moreover, using the \cite[Theorem 1.1.]{zywina} we see that there exists a rational number $u$ such that 

\[j(E) = u^2 + 1728.\]
The Lemma \ref{lem-j-k_isog} implies that $E$ has a rational $5$-isogeny, which means that $G_{\Q}(5)$ is conjugate subgroup of 

\[\left\lbrace\begin{pmatrix}a & b \\ 0 & c\end{pmatrix} : a, c \in \F_5^\times, b \in \F_5\right\rbrace.\]
Using the \cite[Theorem 1.4.]{zywina} we see that there exists a rational number $v \neq 0$ such that 

\[j(E) = \frac{5^2(v^2 + 10v + 5)^3}{v^5}.\]
Therefore, we get the modular curve 

\[C\ :\ 25(v^2 + 10v + 5)^3 - u^2v^5 - 1728v^5 = 0.\]
Computation in \verb|magma| \cite{magma} shows that $C$ is birrationally equivalent to an elliptic curve with Cremona label \href{http://www.lmfdb.org/EllipticCurve/Q/20a3}{20a3}. That curve has only two rational points and both of those points do not correspond to affine points on $C$.
\end{proof}
\begin{lemma}\label{lem:kompozitum_12}
Let $E/\Q$ be an elliptic curve such that $E(\KK)$ contains a point of order $12$. Then 
\[E(\KK)_\tors \simeq \Z/12\Z.\]
\end{lemma}
\begin{proof}
The only other possibility for $E(\KK)_\tors$ is $\Z/2\Z \oplus \Z/12\Z$. Assume that $E(\KK)_\tors \simeq \Z/2\Z \oplus \Z/12\Z$.
The Lemma \ref{lem-j-k_isog} implies that $E$ has a rational $2$-isogeny, which means that $G_{\Q}(2) = G_1$ or $G_{\Q}(2) \sim G_2$ (up to conjugacy), where $G_1$ and $G_2$ are the same as in the proof of Lemma \ref{lem:kompozitum_10}. Therefore we conclude that the entire $2$-torsion subgroup of $E$ is defined over $\Q[2]$.
By the Lemma \ref{lem-j-k_isog} we have that $E$ has a rational $3$-isogeny, which means that the point of order $3$ is defined over at most quadratic extension of $\Q$, so it's defined over $\Q[2]$. From the Proposition \ref{pro:gonzalo_najman} we have that the point of order $4$ is also defined over $\Q[2]$.
Finally, we have that $E(\Q[2])_\tors \simeq \Z/2\Z \oplus \Z/12\Z$, which contradicts the Theorem \ref{p2result}.
\end{proof}
\begin{lemma}\label{lem:kompozitum_13}
There exists elliptic curves $E/\Q$ such that \[E(\Q)_\tors = \{0\} \qquad \text{and} \qquad E(\KK)_\tors \simeq \Z/13\Z.\]
\end{lemma}
\begin{proof}
Assume that $P \in E(\KK)$ is a point of order $13$. We know that $P \notin E(\Q)$ and we know that by the Lemma \ref{lem-j-k_isog}, $E$ has a rational $13$-isogeny. Assume that $E(\KK)$ contains a point $Q$ of order $q$, for some prime number $q \neq 13$. It would then follow that $E$ has a rational $13q$-isogeny, which is impossible by the Corollary \ref{kor:ubij_torziju}. Same reasoning shows that $E$ can't have a point of order $13^2$. Therefore, it remains to find an example of elliptic curve $E/\Q$ that has a point of order $13$ defined over $\KK$. Searching the LMFDB \cite{lmfdb} database using \verb|magma| \cite{magma} we find two elliptic curves defined over $\Q$ with such property:
\begin{align*}
\text{\href{http://www.lmfdb.org/EllipticCurve/Q/20736c1}{20736c1}}\ : \qquad & y^2 = x^3 + 6x + 8, \\
\text{\href{http://www.lmfdb.org/EllipticCurve/Q/20736d1}{20736d1}}\ : \qquad & y^2 = x^3 + 24x + 64. \qedhere
\end{align*}
\end{proof}
\section{Proof of Theorem \ref{teo:rast_qzetap}}
In order to prove this theorem, we will use a few well known results and some technical facts that we will now mention. The proof itself follows from the series of Lemmas.
\\
First, let us recall the Mazur's famous isogeny theorem (Theorem \ref{teo:mazur_izogenije}). Let $E/\Q$ be an elliptic curve with a rational $n$-isogeny. Then $n$ has to be one of the values from the table below. We shall also need the values $\phi(n)$.
\begin{table}[ht]
\centering
\begin{tabular}{|c||c|c|c|c|c|c|c|c|c|c|c|c|c|c|c|} \hline
$n$ & $2$ & $3$ & $4$ & $5$ & $6$ & $7$ & $8$ & $9$ & $10$ & $\boxed{11}$ & $12$ & $13$ & $14$ & $15$ & $16$ \\ \hline
$\phi(n)$ & $1$ & $2$ & $2$ & $2^2$ & $2$ & $2 \cdot 3$ & $2^2$ & $2 \cdot 3$ & $2^2$ & $2 \cdot 5$ & $2^2$ & $2^2 \cdot 3$ & $2 \cdot 3$ & $2^3$ & $2^3$ \\ \hline
\end{tabular}
\vskip 5pt
\begin{tabular}{|c||c|c|c|c|c|c|c|c|c|c|} \hline
$n$ & $\boxed{17}$ & $18$ & $19$ & $21$ & $\boxed{25}$ & $27$ & $37$ & $\boxed{43}$ & $\boxed{67}$ & $\boxed{163}$ \\ \hline 
$\phi(n)$ & $2^4$ & $2 \cdot 3$ & $2 \cdot 3^2$ & $2^2 \cdot 3$ & $2^2 \cdot 5$ & $2 \cdot 3^2$ & $2^2 \cdot 3^2$ & $2 \cdot 3 \cdot 7$ & $2 \cdot 3 \cdot 11$ & $2 \cdot 3^4$ \\ \hline
\end{tabular}
\caption{Possible rational $n$-isogenies together with values $\phi(n)$.}
\label{tablica:svi_moguci_n-ovi}
\end{table}

Numbers $n$ that are contained in a $\boxed{\phantom{1}}$ are the ones that will require special care. Soon it will be obvious why is that the case.

We will also need the following theorem:
\\
\cite[Theorem 1.1]{gonloz}:
\begin{tm}\label{teo:gonloz_abel}
Let $E/\Q$ be an elliptic curve and let $n \geq 2$ be a natural number.

If $\Q(E[n]) = \Qz[n]$, then $n \in \{2, 3, 4, 5\}$. Furthermore, if the extension $\Q(E[n]) / \Q$ is Abelian, then $n \in \{2, 3, 4, 5, 6, 8\}$ and the group $G = \Gal[\Q(E[n])][\Q]$ is isomorphic to one of the following groups:

\begin{center}
\begin{tabular}{|c||c|c|c|c|c|c|} \hline
$n$ & $2$ & $3$ & $4$ & $5$ & $6$ & $8$ \\ \hline
$G$ & \begin{tabular}{c}
$\{0\}$ \\ $\Z/2\Z$ \\ $\Z/3\Z$
\end{tabular} & \begin{tabular}{c}
$\Z/2\Z$ \\ $(\Z/2\Z)^2$
\end{tabular} & \begin{tabular}{c}
$\Z/2\Z$ \\ $(\Z/2\Z)^2$ \\ $(\Z/2\Z)^3$ \\ $(\Z/2\Z)^4$
\end{tabular} & \begin{tabular}{c}
$\Z/4\Z$ \\ $\Z/2\Z \times \Z/4\Z$ \\ $(\Z/4\Z)^2$
\end{tabular} & \begin{tabular}{c}
$(\Z/2\Z)^2$ \\ $(\Z/2\Z)^3$
\end{tabular} & \begin{tabular}{c}
$(\Z/2\Z)^4$ \\ $(\Z/2\Z)^5$ \\ $(\Z/2\Z)^6$
\end{tabular} \\ \hline
\end{tabular}
\end{center}
\end{tm}

The following Corollary makes things easier.

\begin{corollary}\label{kor:gonloz_abel}
Let $E/\Q$ be an elliptic curve. Let $p \geq 3$ be a prime number and let $n \geq 2$ be a natural number. If $\Q(E[n]) \subseteq \Qz$, then $n \in \{2, 3, 4, 5\}$ and 
\[\Q(E[n]) \subseteq \Qz[3^3], \text{ for } p = 3 \quad \text{and} \quad \Q(E[n]) \subseteq \Qz[p], \text{ for } p \geq 5.\]
\end{corollary}
\begin{proof}
We know that $\Qz / \Q$ is cyclic, which means that $\Q(E[n]) / \Q$ is also cyclic. From the Theorem \ref{teo:gonloz_abel}, we see that the only possibilities for $n$ are $n \in \{2, 3, 4, 5\}$ and that $[\Q(E[n]) : \Q] \in \{1, 2, 3, 4\}$, which concludes the proof.
\end{proof}
\begin{proposition}\label{pro:naj_enri_growth}
Let $E/\Q$ be an elliptic curve and let $P \in E(\Qbar)$ be a point of order $2^{n+1}$, for some natural number $n$. Then
\[[\Q(P) : \Q(2P)] \mid 4.\]
\end{proposition}
\begin{corollary}\label{kor:naj_enri_growth}
Let $E/\Q$ be an elliptic curve and $P \in E(\Qz)$ be a point of order $2^n$, for some natural number $n$ and $p \geq 3$ a prime number. Then
\[\Q(P) \subseteq \Qz[3^3], \text{ for } p = 3 \quad \text{and} \quad \Q(P) \subseteq \Qz[p], \text{ for } p \geq 5.\]
\[\Q(P) \subseteq \Qz[3^3], \text{ za } p = 3 \quad \text{and} \quad \Q(P) \subseteq \Qz[p], \text{ za } p \geq 5.\]
\end{corollary}
\begin{proof}
The point $2^{n-1}P$ is of order $2$ and we know that $[\Q(2^{n-1}P) : \Q] \in \{1, 2, 3\}$. Applying the Proposition \ref{pro:naj_enri_growth} $(n-1)$ times we conclude that $[\Q(P) : \Q(2^{n-1}P)] = 2^a$, for some $a \in \{0, 1, 2, \dotsc, 2n - 2\}$, from which it follows that $[\Q(P) : \Q] = 2^a3^b$, for some $a \in \{0, 1, 2, \dotsc, 2n - 1\}$, $b \in \{0, 1\}$ and the claim follows.
\end{proof}
\begin{lemma}\label{lem:sluaj_p_veci_od_13}
Let $E/\Q$ be an elliptic curve and let $p \geq 13$ be a prime number. Then we have
\[E(\Qz)_\tors = E(\Qz[p])_\tors.\]
\end{lemma}
\begin{proof}
Let $q$ be a prime number different then $2$ and $p$. We know that $\Qz$ does not contain $\zeta_q$. By the Proposition \ref{pro:weil} we conclude that $E(\Qz)$ does not contain $E[q]$. Assume that $n$ is a natural number such that $E(\Qz)$ contains a point $P$ of order $q^n$. By the Lemma \ref{lem-j-k_isog} we see that $E$ has a rational $q^n$-isogeny. The Corollary \ref{kor:izogenija_stupanj} implies that $[\Q(P) : \Q] \mid \phi(q^n)$. By the Table \ref{tablica:svi_moguci_n-ovi} we see that $\Q(P) \subseteq \Qz[p]$. Let us now consider the case $q=2$. By the Corollary \ref{kor:naj_enri_growth} we see that the points of order $2^n$ contained in $E(\Qz)_\tors$ are defined over $\Qz[p]$. Finally, using the Corollary \ref{kor:gonloz_abel} we conclude that $E(\Qz)$ does not contain $E[p]$. Therefore, when a point $P \in E(\Qz)$ is of order $p^n$ the same argument applies as in the case when the order equals $q^n$. 
\end{proof}
\begin{lemma}\label{lem:sluaj_p_jednak_11}
Let $E/\Q$ be an elliptic curve. Then
\[E(\Qz[11][\infty])_\tors = E(\Qz[11])_\tors.\]

\end{lemma}
\begin{proof}

The proof is analogous to the proof of the previous Lemma, except in the case when $P \in E(\Qz[11][\infty])$ is of order $67$. In this case $E$ has a rational $67$-isogeny. This means that $j(E) = -2^{15} \cdot 3^3 \cdot 5^3 \cdot 11^3$. From the proof of the Lemma \ref{lem:lagani_prosti} we know that the $67$\textsuperscript{th} division polynomial of elliptic curves with this $j$-invariant has one irreducible factor of degree $\leq 66$ and that this factor, call is $\varphi$, is of degree $33$. This is impossible since $\Qz[11][\infty]$ does not contain a subfield of degree $33$ over $\Q$.
\end{proof}
\begin{lemma}\label{lem:sluaj_p_jednak_7}
Let $E/\Q$ be an elliptic curve. Then \[E(\Qz[7][\infty])_\tors = E(\Qz[7])_\tors.\]
\end{lemma}
\begin{proof}
Let $P$ be a point of order $n$ in $E(\Qz[7][\infty])$. If $n \neq 43$, then the same argument as in the previous lemma applies. Let us consider the case when $n=43$. Using Table \ref{tablica:svi_moguci_n-ovi} and \verb|magma| \cite{magma} we can easily show that such a point does not exist. The proof is basically the same as in the previous Lemma \ref{lem:sluaj_p_jednak_11}.
\end{proof}

\begin{lemma}\label{lem:sluaj_p_jednak_5}
Let $E/\Q$ be an elliptic curve. Then
\[E(\Qz[5][\infty])_\tors = E(\Qz[5])_\tors.\]
\end{lemma}
\begin{proof}
If $E(\Qz[5][\infty])$ contains a point of order $11$, then $E(\Qz[5][\infty])_\tors=\Z/11\Z$. Therefore, $E$ has a rational $11$-isogeny. By Table \ref{tablica:polinomcici} we see that the $11$\textsuperscript{th} division polynomial of elliptic curve $E/\Q$ with a rational $11$-isogeny has a unique irreducible factor of degree $\leq 10$ and that factor is of degree $5$. In the proof of Lemma \ref{lem:lagani_prosti} we have shown that such irreducible factor does not have zeroes over the unique subfield $\Q[1][5]$ of $\Qz[5][\infty]$ of degree $5$ over $\Q$. It remains to consider the case when $E(\Qz[5][\infty])$ contains a point of order $25$. Assume that $P \in E(\Qz[5][\infty])$ is a point of order $25$. We can easily see that $E(\Qz[5][\infty])_\tors \simeq \Z/25\Z$ and so $E$ has a rational $25$-isogeny. This means that $\Q(P) \subseteq \Qz[25]$. Assume that $\Q(P) \nsubseteq \Qz[5]$ (otherwise, we're immediately done). By the analogous argument we conclude that $\Q(5P) \subseteq \Qz[5]$. Since every $\sigma \in \Gal[\Q(E[25])][\Qz[5]]$ fixes a point $5P$, we conclude that $G_{\Qz[5]}(25)$ is (up to conjugacy) of the form  \[\left\lbrace\begin{pmatrix}a & \ast \\ 0 & \ast\end{pmatrix} : a \in 1 + 5\Z/25\Z\right\rbrace.\] By the \cite[Lemma 3.4.]{CDKN} we have $\det G_{\Qz[25]}(25) \simeq \{1\}$. Recall that the group $\Gal[\Qbar][\Qz[25]]$ fixes point $P$ and so \[G_{\Qz[25]}(25) \leq \left\lbrace\begin{pmatrix}1 & \ast \\ 0 & 1\end{pmatrix}\right\rbrace.\] Notice that $\Gal[\Q(E[25])][\Qz[5]]$ does not fix the point $P$ and that \[[G_{\Qz[5]}(25) : G_{\Qz[25]}(25)] = [\Qz[25] : \Qz[5]] = 5.\] Therefore we have \[G_{\Qz[5]}(25) \leq \left\lbrace\begin{pmatrix}a & \ast \\ 0 & 1\end{pmatrix} : a \in 1 + 5\Z/25\Z\right\rbrace.\] We also have $[G_{\Q}(25) : G_{\Qz[5]}(25)] = [\Qz[5] : \Q] = 4$ and \[[(\Z/25\Z)^\times : 1 + 5\Z/25\Z] = 4 \qquad \text{i} \qquad [\diam{7} : \{1\}] = 4.\] Finally we conclude that \[G_{\Q}(25) \leq \left\lbrace\begin{pmatrix}a & \ast \\ 0 & b\end{pmatrix} : a \in (\Z/25\Z)^\times,\ b \in \{7, -1, -7, 1\}\right\rbrace\] and we have \[25 \mid 150 \mid [\GL_2(\Z/25\Z) : G_{\Q}(25)] \mid [\Aut_{\Z_5}(T_5(E)) : \Img(\overline{\rho}_{5, E})],\] which contradicts the \cite[Theorem 2]{greenberg}.
\end{proof}
\begin{lemma}\label{lem:sluaj_p_jednak_3}
Let $E/\Q$ be an elliptic curve. Then \[E(\Qz[3][\infty])_\tors = E(\Qz[3][3])_\tors.\]
\end{lemma}
\begin{proof}
If $q \neq 3$ is a prime number and $n$ is a natural number such that $E(\Qz[3][\infty])$ contains a point $P$ of order $q^n$, then as in the proofs of the previous lemmas we see that $\Q(P) \subseteq \Qz[3][3]$, except when $q=163$ and $n=1$. Assume that $P \in E(\Qz[3][\infty])$ is the point of order $163$. We know that the only possibility is $E(\Qz[3][\infty])_\tors \simeq \Z/163\Z$, which implies that $E$ has a rational $163$-isogeny. Furthermore, we have $\Q(P) \subseteq \Qz[3][5]$. By Table \ref{tablica:polinomcici} we see that the $163$\textsuperscript{rd} division polynomial of $E$ has only one irreducible factor of degree $\leq 162$ and that such factor (call it $\psi$) has degree $81$. This means that the polynomial $\psi$ needs to have zeroes defined over the field $\Qz[3][5]$. More precisely, it needs to have zeroes defined over a number field of degree $81$ that is contained in $\Qz[3][5]$. Such a field is unique and it is equal to $\Q[4][3]$. In the proof of the Lemma \ref{lem:lagani_prosti} we have shown that this is in fact false. Therefore we have arrived at the contradiction. We conclude that $E(\Qz[3][\infty])$ can't contain a point of order $163$.
It remains to consider the points of order $3^n$. We see that $E(\Qz[3][\infty])[3^\infty]$ is isomorphic to one of the following groups: \[\Z/3\Z, \quad \Z/9\Z, \quad \Z/27\Z, \quad \Z/3\Z \oplus \Z/3\Z, \quad \Z/3\Z \oplus \Z/9\Z.\] If $E(\Qz[3][\infty])[3^\infty] \simeq \Z/3^n\Z$, for some $n \in \{1, 2, 3\}$, then we know that $E$ has a rational $3^n$-isogeny and the claim of the Lemma follows immediately by the Table \ref{tablica:svi_moguci_n-ovi}. The case when $E(\Qz[3][\infty])[3^\infty] \simeq \Z/3\Z \oplus \Z/3\Z$ follows from the Corollary \ref{kor:gonloz_abel}. Let us check the remaining possibility, when $E(\Qz[3][\infty])[3^\infty] \simeq \Z/3\Z \oplus \Z/9\Z$. Let $\{ P, Q\}$ be the basis for $E(\Qz[3][\infty])[3^\infty]$. The point $P$ is of order $3$ and the point $Q$ is of order $9$. By the Theorem \ref{teo:gonloz_abel}, we know that $\Q(P) = \Q(3Q) = \Qz[3]$. Proposition \cite[Proposition 4.6.]{gn} implies that \[[\Q(Q) : \Q(3Q)] \in \{1, 2, 3, 6, 9\},\] which shows that $\Q(P, Q) \subseteq \Qz[3][3]$.
\end{proof}
Before proving that $E(\Qz[2][\infty])_\tors = E(\Qz[2][4])_\tors$, we shall state the two techical results we shall need.
\begin{lemma}\label{lem:sluaj_p_jednak_2_4+8}
Let $E/\Q$ be an elliptic curve such that $E(\Qz[2][\infty])[2^\infty] \simeq \Z/4\Z \oplus \Z/8\Z$. Then
\[E(\Qz[2][4])[2^\infty] \simeq \Z/4\Z \oplus \Z/8\Z.\]
\end{lemma}
\begin{proof}
Let $P$ and $Q$ be a points in $E(\Qz[2][\infty])$ of orders $4$ and $8$ respectively that generate $\Z/4\Z \oplus \Z/8\Z$ torsion subgroup. We need to show that $\Q(P, Q) \subseteq \Q(\zeta_{2^4})$. By the Theorem \ref{teo:gonloz_abel} we see that $\Q(P, 2Q) \subseteq \Q(\zeta_{2^3})$. Using the Proposition \ref{pro:naj_enri_growth} we can see that $\Q(P, Q) \subseteq \Q(\zeta_{2^5})$. In \cite{derdrew} M.\ Derickx and A.\ V.\ Sutherland have computed the models for a certain number of modular curves $X_1(m, mn)$. This list can be found in \url{http://math.mit.edu/~drew/X1mn.html}. We observe that $X_1(4, 8)$ is actually an elliptic curve \href{http://www.lmfdb.org/EllipticCurve/Q/32a2}{32a2} (data \cite{lmfdb}):\[E'\ :\ y^2 = x^3 - x.\] We want to show that $r(E'(\Qz[2][5])) = r(E(\Qz[2][4])) = 0$ and $E'(\Qz[2][5])_\tors = E(\Qz[2][4])_\tors$. We shall use the fact that \[r(E'(\Qz[2][5])) = r(E'(\Q[3][2])) + r(E'^{(-1)}(\Q[3][2])).\] This is easily seen to be true because $\Qz[2][5] = \Q[3][2]\Q(i)$. Computation in \verb|magma| \cite{magma} shows that this claim holds.
\end{proof}
\begin{lemma}\label{lem:slucaj_p_jednak_2_CM}
Let $E/\Q$ be an elliptic curve with CM. Then $E$ does not contain a point of order $16$ defined over $\Qz[2][\infty]$.
\end{lemma}
\begin{proof}
Assume that $E(\Qz[2][\infty])$ contains a point $P$ of order $16$. We see that $E(\Qz[2][\infty])_\tors$ is isomorphic to one of the groups \[\Z/16\Z, \quad \Z/2\Z \oplus \Z/16\Z, \quad \Z/4\Z \oplus \Z/16\Z.\] In the first scenario, by the Table  \ref{tablica:svi_moguci_n-ovi} we see that $\Q(P) \subseteq \Qz[2][4]$. In the second case ($\Z/2\Z \oplus \Z/16\Z$), we know that $\diam{2P}$ is the kernel of a rational $8$-isogeny and so $\Q(2P) \subseteq \Qz[2][3]$. Using the Proposition \ref{pro:naj_enri_growth} we conclude that $\Q(P) \subseteq \Qz[2][5]$. Finally, consider the third case. The group $\diam{4P}$ is the kernel of a rational $4$-isogeny and so $\Q(4P) \subseteq \Qz[2][2]$. By the Proposition \cite[Proposition 4.6.]{gn}, it follows that $\Q(P) \subseteq \Qz[2][6]$. In any of these cases, the point $P$ is defined over $\Qz[2][6]$. Since $E$ has CM, $j$-invariant of $E$ is equal to one of the following $13$ values:  \[-262537412640768000,\ -147197952000,\ -884736000,\ -12288000,\ -884736,\] \[-32768,\ -3375,\ 0,\ 1728,\ 8000,\ 54000,\ 287496,\ 16581375.\] Computation in \verb|magma| \cite{magma} shows that the $8$\textsuperscript{th} division polynomial of elliptic curves with those $j$-invariants does not have zeroes over $\Qz[2][6]$, which completes the proof.
\end{proof}
\begin{lemma}\label{lem:sluaj_p_jednak_2_2+16}
Let $E/\Q$ be an elliptic curve such that $E(\Qz[2][\infty])[2^\infty] \simeq \Z/2\Z \oplus \Z/16\Z$. Then \[E(\Qz[2][4])[2^\infty] \simeq \Z/2\Z \oplus \Z/16\Z.\]
\end{lemma}
\begin{proof}
Let $P$ and $Q$ be the points in $E(\Qz[2][\infty])$ of orders $2$ and $16$ respectively, that generate the $\Z/2\Z \oplus \Z/16\Z$ torsion subgroup. We need to show that $\Q(P, Q) \subseteq \Q(\zeta_{2^4})$. We can see that $\langle 2Q \rangle$ is the kernel of the rational $8$-isogeny, which means that $2Q$ is defined over a number field of degree at most $4$. By the Proposition \ref{pro:naj_enri_growth}, we see that $Q$ is defined over a number field of degree at most $16$. Therefore, $\Q(P, Q) \subseteq \Q(\zeta_{2^5})$. Let us show the following claim:
\\
If $[\Q(P, Q):\Q] \geq 16$, then $\Q(P, Q)$ is not contained in $\Qz[2][\infty]$. This will imply the claim of our Lemma. We know that $\Gal[\Q(\zeta_2)][\Q] \simeq \{0\}$, $\Gal[\Q(\zeta_4)][\Q] \simeq \Z/2\Z$ and \[\Gal[\Q(\zeta_{2^k})][\Q] \simeq \Z/2\Z \oplus \Z/2^{k-2}\Z,\] for every positive integer $k>2$. In the case when $[\Q(P, Q):\Q] \geq 16$, we have that $\Gal[\Q(P, Q)][\Q]$ is not of this form and the claim follows. By the previous Lemma \ref{lem:slucaj_p_jednak_2_CM} we know that $E$ does not have CM. In the proof of \cite[Lemma 8.15.]{gn}, E.\ Gonz\'{a}lez-Jim\'{e}nez and F.\ Najman have used \verb|magma| \cite{magma} \href{http://verso.mat.uam.es/~enrique.gonzalez.jimenez/research/tables/growth/lem8_16a.txt}{code} (\url{http://verso.mat.uam.es/~enrique.gonzalez.jimenez/research/tables/growth/growth.html}). We shall do the same. Examining each possible $2$-adic representation, we find all possibilities for $\Gal[\Q(P, Q)][\Q]$, under the assumption that $[\Q(P, Q) : \Q] \geq 16$. Let us mention that the key to this approach is the data available at \href{http://verso.mat.uam.es/~enrique.gonzalez.jimenez/research/tables/pprimary/RZB-2adic/2primary_Ss.txt}{2primary\_Ss.txt} which was computed by E.\ Gonz\'{a}lez-Jim\'{e}nez and \'{A}.\ Lozano-Robledo. All \verb|magma| \cite{magma} codes can be found on \url{http://verso.mat.uam.es/~enrique.gonzalez.jimenez/research/tables/growth/growth.html}. Finally, a calculation in \verb|magma| \cite{magma} shows that the field $\K$ such that $\Q(P, Q) = \K$ and $[\K : \Q] \geq 16$ satisfies $\Gal[\K][\Q] \simeq (\Z/2\Z)^2 \oplus \Z/4\Z$. This means that $\K$ is not contained in $\Qz[2][\infty]$, which completes the proof. 
\end{proof}
\begin{lemma}\label{lem:sluaj_p_jednak_2_4+16}
Let $E/\Q$ be an elliptic curve such that $E(\Qz[2][\infty])[2^\infty] \simeq \Z/4\Z \oplus \Z/16\Z$. Then
\[E(\Qz[2][4])[2^\infty] \simeq \Z/4\Z \oplus \Z/16\Z.\]
\end{lemma}
\begin{proof}
The proof is similar to the proof of the previous Lemma. If $P$ and $Q$ are the points in $E(\Qz[2][\infty])$ of orders $4$ and $16$ respectively that generate torsion subgroup $\Z/4\Z \oplus \Z/16\Z$, then $\langle 4Q \rangle$ is the kernel of the rational $4$-isogeny and $\Q(4Q)$ is at most quadratic extension of $\Q$. We conclude that the degree of $\Q(P, Q)$ is at most $32$. Therefore,  $\Q(P, Q) \subseteq \Qz[2][6]$. A calculation in \verb|magma| \cite{magma} shows that if $\K = \Q(P, Q)$ and $[\K : \Q] \geq 16$, then $\Gal[\K][\Q]$  is isomorphic to one of the groups \[(\Z/2\Z)^4, \quad (\Z/2\Z)^2 \oplus \Z/4\Z, \quad (\Z/2\Z)^5, \quad (\Z/2\Z)^3 \oplus \Z/4\Z.\] We conclude that $\K$ is not contained in $\Qz[2][\infty]$, which completes the proof.
\end{proof}
Finally, the following Lemma concludes the proof of Theorem \ref{teo:rast_qzetap}.
\begin{lemma}\label{lem:sluaj_p_jednak_2}
Let $E/\Q$ be an elliptic curve. Then
\[E(\Qz[2][\infty])_\tors = E(\Qz[2][4])_\tors.\]
\end{lemma}
\begin{proof}
Let $q > 2$ be a prime number and $n$ a positive integer such that $E(\Qz[2][\infty])$ contains a point $P$ of order $q^n$. As before, we conclude that $\Q(P) \subseteq \Qz[2][4]$, except when $q=17$ and $n=1$. Let $P \in E(\Qz[2][\infty])$ be a point of order $17$. We know that $E(\Qz[2][\infty])_\tors \simeq \Z/17\Z$, which means that $E$ has a rational $17$-isogeny. Furthermore, we have $\Q(P) \subseteq \Qz[2][5]$. By \cite{loz} and \cite{lmfdb}, we see that
\begin{align*}
j(E) = \frac{-17^2 \cdot 101^3}{2}, &\quad \text{which is true for elliptic curve \href{https://www.lmfdb.org/EllipticCurve/Q/14450p1/}{14450p1}} \\
	&\text{or} \\
j(E) = \frac{-17 \cdot 373^3}{2^{17}}, &\quad \text{which is true for elliptic curve \href{https://www.lmfdb.org/EllipticCurve/Q/14450p2/}{14450p2}.}
\end{align*}
For each of those two $j$-invariants we compute and factor the $17$\textsuperscript{th} division polynomials associated to the two elliptic curves previously mentioned. More precisely, we're searching for irreducible factors of degree $\leq 16$ and we check that no such factors have zeroes defined over $\Qz[2][5]$. Therefore, $E(\Qz[2][\infty])$ can't contain a point of order $17$.

It remains to consider the points of order $2^n$. We see that $E(\Qz[2][\infty])[2^\infty]$ is isomorphic to one of the following groups:
\[\Z/2\Z, \quad \Z/4\Z, \quad \Z/8\Z, \quad \Z/16\Z,\]
\[\Z/2\Z \oplus \Z/2\Z, \quad \Z/2\Z \oplus \Z/4\Z, \quad \Z/2\Z \oplus \Z/8\Z, \quad \Z/2\Z \oplus \Z/16\Z,\]
\[\Z/4\Z \oplus \Z/4\Z, \quad \Z/4\Z \oplus \Z/8\Z, \quad \Z/4\Z \oplus \Z/16\Z,\]
\[\Z/8\Z \oplus \Z/8\Z.\]
If $E(\Qz[2][\infty])[2^\infty] \simeq \Z/2^n\Z$, for some $n \in \{1, 2, 3, 4\}$, then $E$ has a rational $2^n$-isogeny so the claim immediately follows from the Table \ref{tablica:svi_moguci_n-ovi}.

We know that $\Gal[\Q(\zeta_2)][\Q] \simeq \{0\}$, $\Gal[\Q(\zeta_4)][\Q] \simeq \Z/2\Z$ and 
\[\Gal[\Q(\zeta_{2^k})][\Q] \simeq \Z/2\Z \oplus \Z/2^{k-2}\Z,\] for every positive integer $k > 2$. Using the Theorem \ref{teo:gonloz_abel} we conclude that the torsion $\Z/8\Z \oplus \Z/8\Z$ can not occur. Furthermore, the same Theorem implies that our claim holds in the case when the torsion subgroup is isomorphic to $\Z/2\Z \oplus \Z/2\Z$ or $\Z/4\Z \oplus \Z/4\Z$.

Lemma \ref{lem:sluaj_p_jednak_2_4+8} implies that if $E(\Qz[2][\infty])[2^\infty] \simeq \Z/4\Z \oplus \Z/8\Z$, then that torsion subgroup is actually defined over $\Qz[2][4]$. When the torsion subgroup equals to $\Z/2\Z \oplus \Z/16\Z$ or $\Z/4\Z \oplus \Z/16\Z$, the claim follows by the Lemma \ref{lem:sluaj_p_jednak_2_2+16} or the Lemma \ref{lem:sluaj_p_jednak_2_4+16}, respectively.

It remains to check the case when the torsion subgroup is $\Z/2\Z \oplus \Z/4\Z$ or $\Z/2\Z \oplus \Z/8\Z$. Assume that the points $P$ and $Q$ generate such torsion subgroup. The group $\diam{2Q}$ is the kernel of a rational $k$-isogeny and we know that $2Q$ is defined over a number field of degree at most $2$. This means that the point $Q$ is defined over a number field of degree at most $2 \cdot 4 = 8$, so it's defined over $\Qz[2][4]$. This completes the proof of this Lemma and the proof of the Theorem \ref{teo:rast_qzetap}.
\end{proof}

\vfill
\begin{remark}
\verb|Magma| \cite{magma} code used in this paper can be found on Ivan Krijan's \href{https://web.math.pmf.unizg.hr/~ikrijan/}{webpage}.
\end{remark}
\begin{acknowledgments}
We would like to thank \href{https://hdaniels.people.amherst.edu/}{Harris B.\ Daniels}, \href{http://www.maartenderickx.nl/}{Maarten Derickx} and \href{https://web.math.pmf.unizg.hr/~fnajman/}{Filip Najman} for great direct and indirect help and very useful discussions.
\end{acknowledgments}

\clearpage
\bibliographystyle{plain}
\bibliography{Zp_bib}

\begin{thebibliography}{10}

\bibitem{magma}
W.~Bosma, J.~Cannon, and C.~Playoust.
\newblock The magma algebra system. i. the user language.
\newblock {\em J. Symbolic Comput.}, 24(3-4):235--265, 1997.

\bibitem{chou}
M.~Chou.
\newblock Torsion of rational elliptic curves over the maximal abelian
  extension of $\mathbb{Q}$.
\newblock {\em Pacific J. Math.}, 302(2):481--509, 2019.

\bibitem{CDKN}
M.~Chou, H.~B. Daniels, I.~Krijan, and F.~Najman.
\newblock Torsion groups of elliptic curves over the $\mathbb{Z}_p$-extensions
  of $\mathbb{Q}$.
\newblock Submitted.

\bibitem{dlns}
H.~B. Daniels, \'{A}. Lozano-Robledo, F.~Najman, and A.~V. Sutherland.
\newblock Torsion points on rational elliptic curves over the compositum of all
  cubic fields.
\newblock {\em Math. Comp.}, 87:425--458, 2018.

\bibitem{derdrew}
M.~Derickx and A.~V. Sutherland.
\newblock Torsion subgroups of elliptic curves over quintic and sextic number
  fields.
\newblock {\em Proc. Amer. Math. Soc.}, 145:4233--4245, 2017.

\bibitem{gonloz}
E.~Gonz\'{a}lez-Jim\'{e}nez and \'{A}. Lozano-Robledo.
\newblock Elliptic curves with abelian division fields.
\newblock {\em Math. Z.}, 283:835–--859, 2016.

\bibitem{gn}
E.~Gonz\'{a}lez-Jim\'{e}nez and F.~Najman.
\newblock Growth of torsion groups of elliptic curves upon base change.
\newblock {\em Math. Comp.}, 89:1457--1485, 2020.

\bibitem{greenberg2}
R.~Greenberg.
\newblock Iwasawa theory for elliptic curves.
\newblock In C.~Viola, editor, {\em Arithmetic Theory of Elliptic Curves},
  volume 1716 of {\em Lecture Notes in Mathematics}, pages 51--144. Springer,
  Berlin, Heidelberg, 1999.

\bibitem{greenberg}
R.~Greenberg.
\newblock The image of galois representations attached to elliptic curves with
  an isogeny.
\newblock {\em Amer. J. Math.}, 134(5):1167--1196, 2012.

\bibitem{kenku2}
M.~A. Kenku.
\newblock The modular curve $x_0(39)$ and rational isogeny.
\newblock {\em Math. Proc. Cambridge Philos. Soc.}, 85:21--23, 1979.

\bibitem{kenku4}
M.~A. Kenku.
\newblock The modular curve $x_0(169)$ and rational isogeny.
\newblock {\em J. London Math. Soc.}, s2-22:239--244, 1980.
\newblock Corrigendum: https://doi.org/10.1112/jlms/s2-23.3.428-s.

\bibitem{kenku3}
M.~A. Kenku.
\newblock The modular curves $x_0(65)$ and $x_0(91)$ and rational isogeny.
\newblock {\em Math. Proc. Cambridge Philos. Soc.}, 87:15--20, 1980.

\bibitem{kenku5}
M.~A. Kenku.
\newblock The modular curve $x_0(125)$, $x_1(25)$ and $x_1(49)$.
\newblock {\em J. London Math. Soc.}, s2-23:415--427, 1981.

\bibitem{loz}
\'{A} Lozano-Robledo.
\newblock On the field of definition of $p$-torsion points on elliptic curves
  over the rationals.
\newblock {\em Math. Ann.}, 357:279--305, 2013.

\bibitem{mazur2}
B.~Mazur.
\newblock Rational isogenies of prime degree.
\newblock {\em Invent. Math.}, 44:129--162, 1978.

\bibitem{najman_twist}
F.~Najman.
\newblock The number of twists with large torsion of an elliptic curve.
\newblock {\em Rev. R. Acad. Cienc. Exactas F\'{\i}s. Nat. Ser. A Mat. RACSAM},
  109:535--547, 2015.

\bibitem{siksek}
S.~Siksek.
\newblock Explicit methods for modular curves, 2019.
\newblock Lecture notes.

\bibitem{silverman}
J.~H. Silverman.
\newblock {\em The Arithmetic of Elliptic Curves}, volume 106 of {\em Graduate
  Texts in Mathematics}.
\newblock Springer-Verlag New York, 2 edition, 2009.

\bibitem{lmfdb}
{The LMFDB Collaboration}.
\newblock The l-functions and modular forms database, 2019.
\newblock [Online; accessed 30 October 2019].

\bibitem{washington}
L.~C. Washington.
\newblock {\em Introduction to Cyclotomic Fields}, volume~83 of {\em Graduate
  Texts in Mathematics}.
\newblock Springer-Verlag New York, 2 edition, 1997.

\bibitem{zywina}
D.~Zywina.
\newblock On the possible images of the mod $\ell$ representations associated
  to elliptic curves over $\mathbb{Q}$.
\newblock Submitted.

\end{thebibliography}
\end{document}